\documentclass{article}

\usepackage[utf8]{inputenc}
\usepackage[english]{babel}
\usepackage{amssymb,amsthm,amsmath,amsfonts,mathrsfs}
\usepackage[table]{xcolor}
\usepackage{tikz}
\usepackage{hyperref}
\usepackage{xfrac}
\usepackage{wrapfig}

\newtheorem{theorem}{Theorem}[section]

\newtheorem{lemma}[theorem]{Lemma}

\theoremstyle{definition}
\newtheorem{definition}[theorem]{Definition}
\newtheorem{example}[theorem]{Example}

\theoremstyle{remark}

\newtheorem{observation}[theorem]{Observation}

\newcommand{\N}{\mathcal{N}}

\newcommand{\tu}{\includegraphics[width=0.075\linewidth]{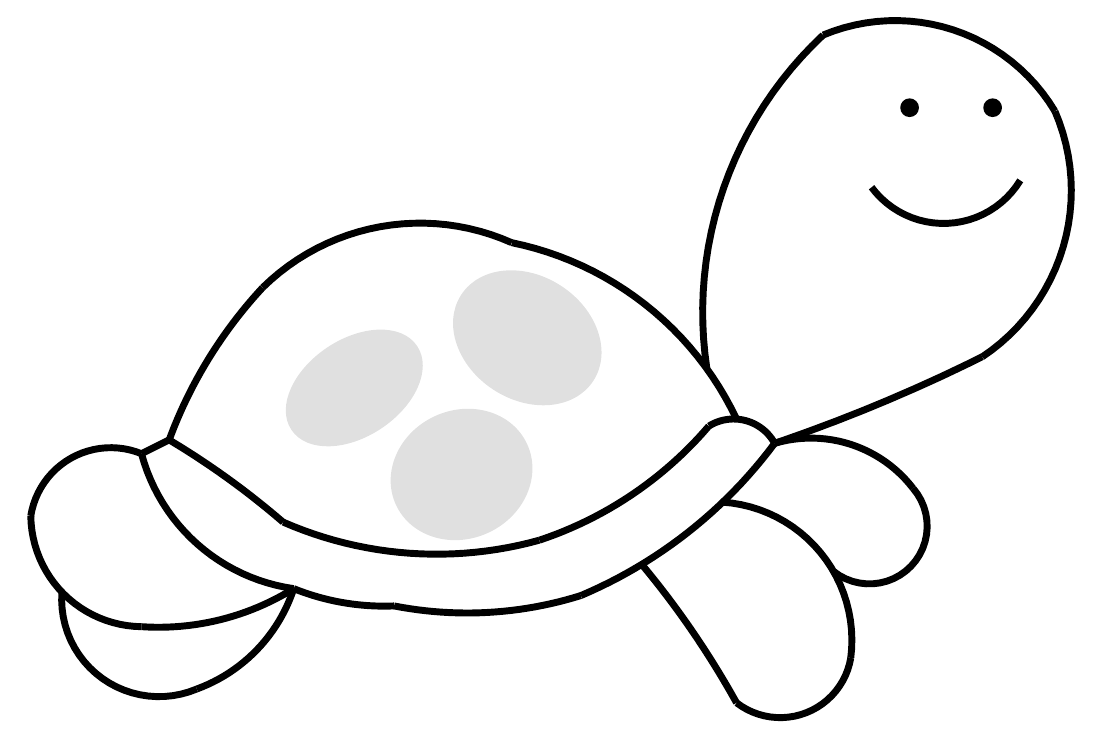}}
\newcommand{\td}{\includegraphics[width=0.075\linewidth]{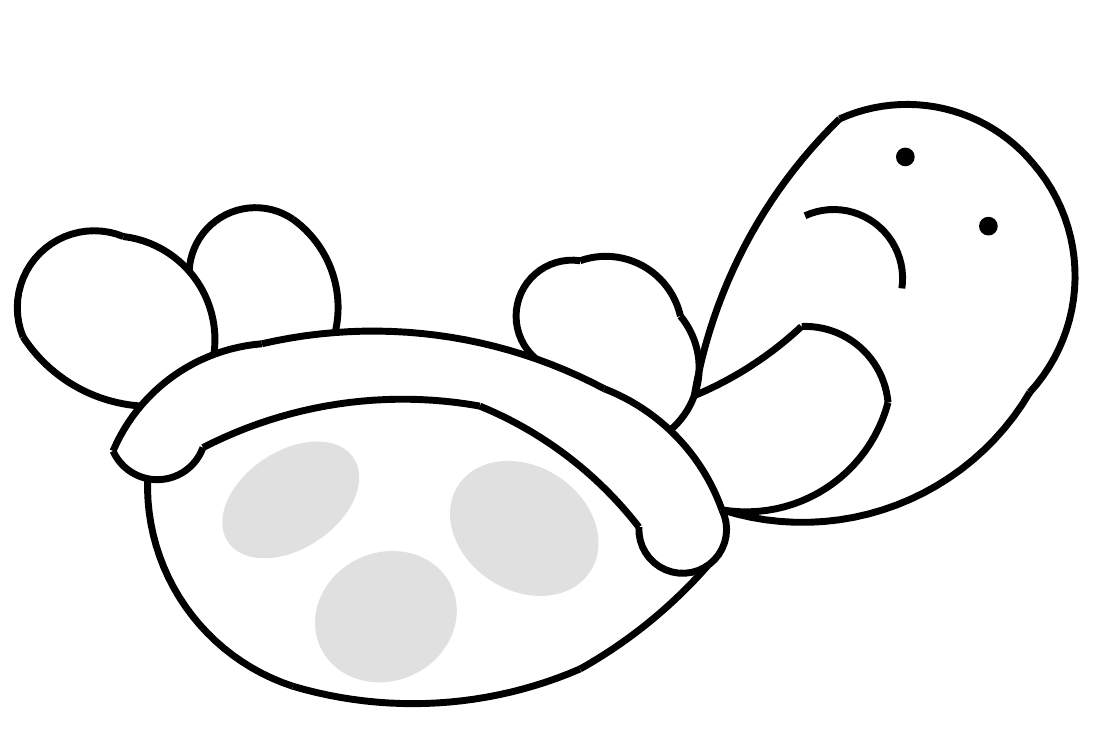}}

\emergencystretch=\maxdimen
\begin{document}

\title{A Note on Numbers}
\maketitle

\begin{center}
  {\sc Alda Carvalho}\textsuperscript{1},
  {\sc Melissa A. Huggan}\textsuperscript{2},
    {\sc Richard J. Nowakowski}\textsuperscript{3},\\
      {\sc Carlos Pereira dos Santos}\textsuperscript{4}
  \par \bigskip

  \textsuperscript{1}ISEL--IPL \& CEMAPRE--University of Lisbon, \href{mailto:alda.carvalho@isel.pt}{alda.carvalho@isel.pt} \par
  \textsuperscript{2}Ryerson University, \href{mailto:melissa.huggan@ryerson.ca}{melissa.huggan@ryerson.ca} \par
          \textsuperscript{3}Dalhousie University, \href{mailto:r.nowakowski@dal.ca}{r.nowakowski@dal.ca} \par
    \textsuperscript{4}Center for Functional Analysis, Linear Structures and Applications, University of Lisbon \& ISEL--IPL, \href{mailto:carlos.santos@isel.pt}{carlos.santos@isel.pt}

   \bigskip

\end{center}

\begin{abstract}
\noindent
When are all positions of a game numbers? We show that two properties are necessary and sufficient. These properties are consequences of that, in a number, it is not an advantage to be the first player. One of these properties implies the other. However, checking for one or the other, rather than just one, can often be accomplished by only looking at the positions on the `board'. If the stronger property holds for all positions, then the values are integers.
\end{abstract}


\section{Introduction}
When analyzing games, an early question is: is it possible that all the positions are numbers? If that is true, then it is easy to determine the outcome of a disjunctive sum of positions, just add up the numbers. It is also easy to find the best move, just play the summand with the largest denominator. The problem is how to recognize when all the positions are numbers.

Siegel \cite{Sie}, page 81, exercise 3.15, states ``If every incentive of $G$ is negative then $G$ is a number''. This does not provide much insight or intuition. In fact, in most non-all-small-games, there are non-zero positions, some of which are  numbers and others not. Let $\mathbf{S}$ be a set of positions of a ruleset. It is called a \emph{hereditary closed set of positions of a ruleset} (HCR) if it is closed under taking options. These HCR sets are the natural objects to consider.

There are two properties either of which, if satisfied for all followers of a position, tells us that the position is a number.
 Both are aspects of the  \textit{first-move-disadvantage} in numbers. The first is a comparison with one move against two moves.

\begin{definition}[F1  Property]\label{def:sp}
Let $\mathbf{S}$ be a HCR. Given $G\in \mathbf{S}$, the pair $(G^L,G^R)\in G^\mathcal{L}\times G^\mathcal{R}$ satisfies the \emph{F1  property} if there is $G^{RL}\in G^{R\mathcal{L}}$ such that $G^{RL}\geqslant G^L$ or there is $G^{LR}\in G^{L\mathcal{R}}$ such that $G^{LR}\leqslant G^R$.
\end{definition}

The second property involves moves by both players.

\begin{definition} [F2  Property]\label{def:dp}
Let $\mathbf{S}$ be a HCR. Given $G\in \mathbf{S}$, $(G^L,G^R)\in G^\mathcal{L}\times G^\mathcal{R}$ satisfies the \emph{F2  property} if there are $G^{LR}\in G^{L\mathcal{R}}$ and $G^{RL}\in G^{R\mathcal{L}}$ such that $G^{RL}
\geqslant G^{LR}$.
\end{definition}

In many games, the literal form of positions will tell us whether they satisfy the F1 property or the F2 property with equality. See Section~\ref{sec:examples} for examples.\\

The two results a player should remember are:\\

\noindent \textbf{Lemma \ref{th:sc}} Let $\mathbf{S}$ be a {\rm HCR}. If, for any position $G\in \mathbf{S}$, all pairs $(G^L,G^R)\in G^\mathcal{L}\times G^\mathcal{R}$ satisfy the F1  property or the F2  property, then  all positions $G\in\mathbf{S}$ are numbers;\\

and\\

\noindent \textbf{Theorem \ref{th:i}}
Let $\mathbf{S}$ be a {\rm HCR}. If, for any position $G\in \mathbf{S}$, all pairs $(G^L,G^R)\in G^\mathcal{L}\times G^\mathcal{R}$ satisfy the F2  property, then all positions $G\in\mathbf{S}$ are integers.\\

Theorem \ref{th:con} is the central theoretical result: All the positions in a HCR set are numbers, if and only if there is no position and no number such that the sum is an $\N$-position. This is all that is required to prove Lemma \ref{th:sc}.
  Lemma \ref{lem:descendfloss} shows that the F2  property implies the F1  property, with strict inequality.
Now, it may seem that the F2  property is useless and it shouldn't be a hypothesis of Lemma \ref{th:sc}. However, in practice, it is easier to recognize that all pairs satisfy either the F1  property or the F2  property rather than trying to prove that all pairs satisfy just the F1  property. In fact, Lemma \ref{th:sc} may be written as a necessary and sufficient condition. This is Theorem \ref{th:snc} which only uses the  F1  property.

The F2  property is a stronger constraint than the F1  property.  Theorem~\ref{th:i} shows that the F2  property implies that the numbers will be integers.
\\

We recall the results about numbers needed for this paper.

\begin{theorem}{\rm\cite{ANW,BCG,Sie}} Let $G$ be a number whose options are numbers.
\begin{enumerate}
\item After removing dominated options, the form of $G$ has at most one Left option and at most one Right option.
\item For the options that exist, $G^L<G<G^R$.
\item If there is an integer $k$, $G^L<k<G^R$, or if either $G^L$ or $G^R$ does not exist, then $G$ is an integer.
\item If both $G^L$ and $G^R$ exist and the previous case does not apply, then $G$ is the simplest number between $G^L$ and $G^R$.
\end{enumerate}
\end{theorem}

The most important point to remember is item 2, that is, when a player plays in a number the situation gets worse for them. This has an important consequence when games are being analyzed.

\begin{theorem}[Number Avoidance Theorem]\label{th:nat}{\rm \cite{ANW,BCG,Sie}} Suppose that $G$ is a number and $H$ is
not. If Left can win moving first on $G+H$, then Left can do so with a move on $H$.
\end{theorem}

In many cases, when checking the properties, the Left and Right options will refer to two specific moves on the `game board', one by Left and one by Right. If this happens, then the actual positions will automatically give the stronger conditions, $G^{LR} \cong G^{RL}$ or $G^{RL} \cong G^L$. Moreover, no calculations are required. Examples are given in Section \ref{sec:examples}.

\section{Examples and a Warning}\label{sec:examples}

In these examples, we illustrate that, sometimes, only two specific moves on the `game board', one for each player, are sufficient. We will refer to the specific moves by lower case letters, $\ell$ for Left  and $r$ for Right.

We first sketch a proof to show that the values of {\sc polychromatic chomp} (see Appendix for rules), {\sc blue-red-hackenbush} strings \cite{BCG,Roode} are numbers, and that \textsc{cutcake}~\cite{BCG} positions are integers.
We then give the properties that the following games satisfy:  {\sc domino shave}~\cite{DS},  {\sc shove}~\cite{ANW}, {\sc push}~\cite{ANW}, {\sc lenres}~\cite{SieAng}, {\sc divisors}, and {\sc partizan turning turtles}~\cite{BHN} (see Appendix for last two rulesets). 
 The two games, \textsc{partizan euclid}~\cite{MN} and \textsc{partizan subtraction}~\cite{Mesdal} are examples where many positions satisfy one or both properties. However, since there are positions which satisfy neither, then only a few positions are numbers.
\begin{example}
Let $G$ be a {\sc polychromatic chomp} position.  Let $\ell$ and $r$ be black and gray squares respectively.
If neither $G^{\ell}$ nor $G^r$ eliminates the other, then playing both moves, in either order, results in the same position $Q$, i.e., $G^{\ell r}\cong G^{r\ell}$.
Suppose $G^\ell$ eliminates $G^r$, as illustrated in Figure \ref{fig3}. In this case, Left can play her move before or after Right's
move, i.e., $G^\ell\cong G^{r\ell}$.
\begin{figure}[hbt]
\begin{center}
\scalebox{0.4}{\includegraphics[width=\linewidth]{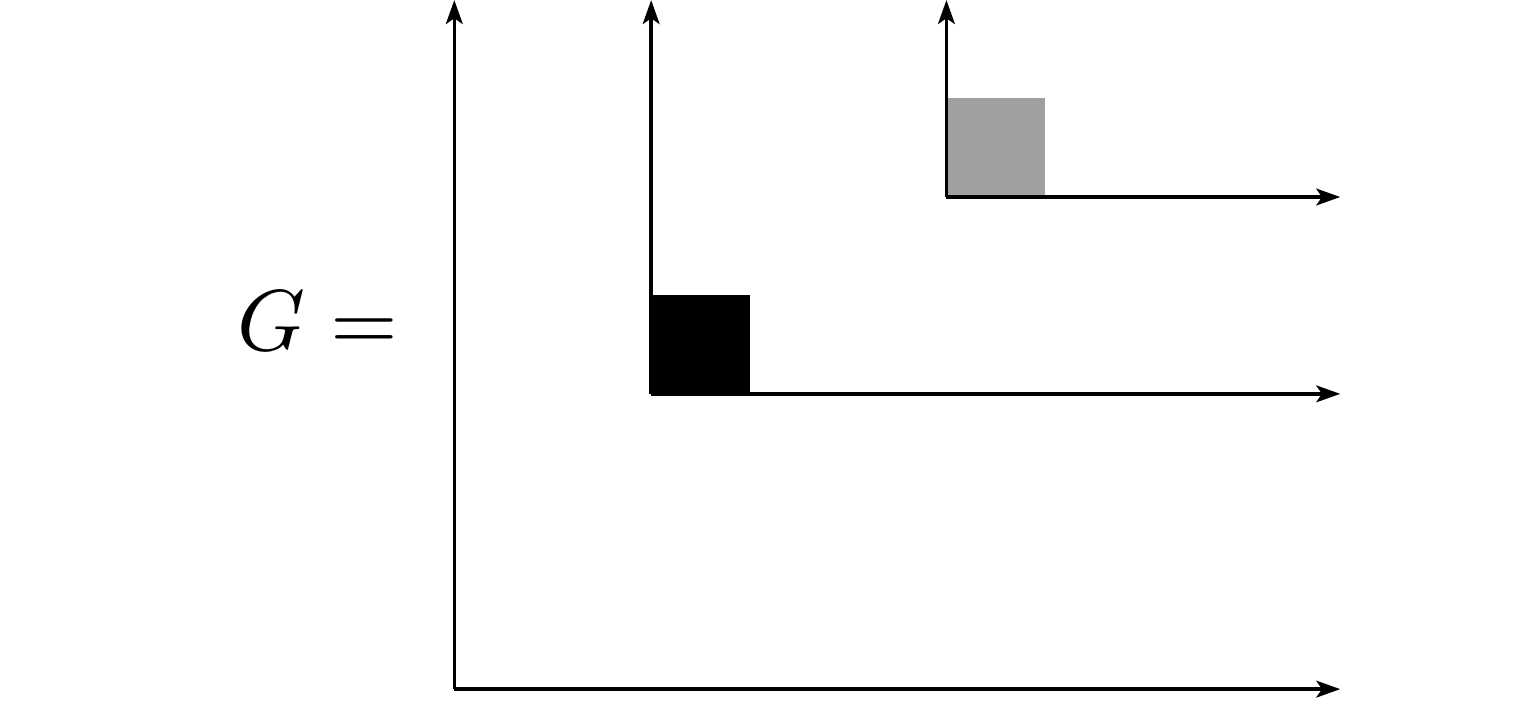}}
\end{center}

\begin{center}
\scalebox{0.4}{\includegraphics[width=\linewidth]{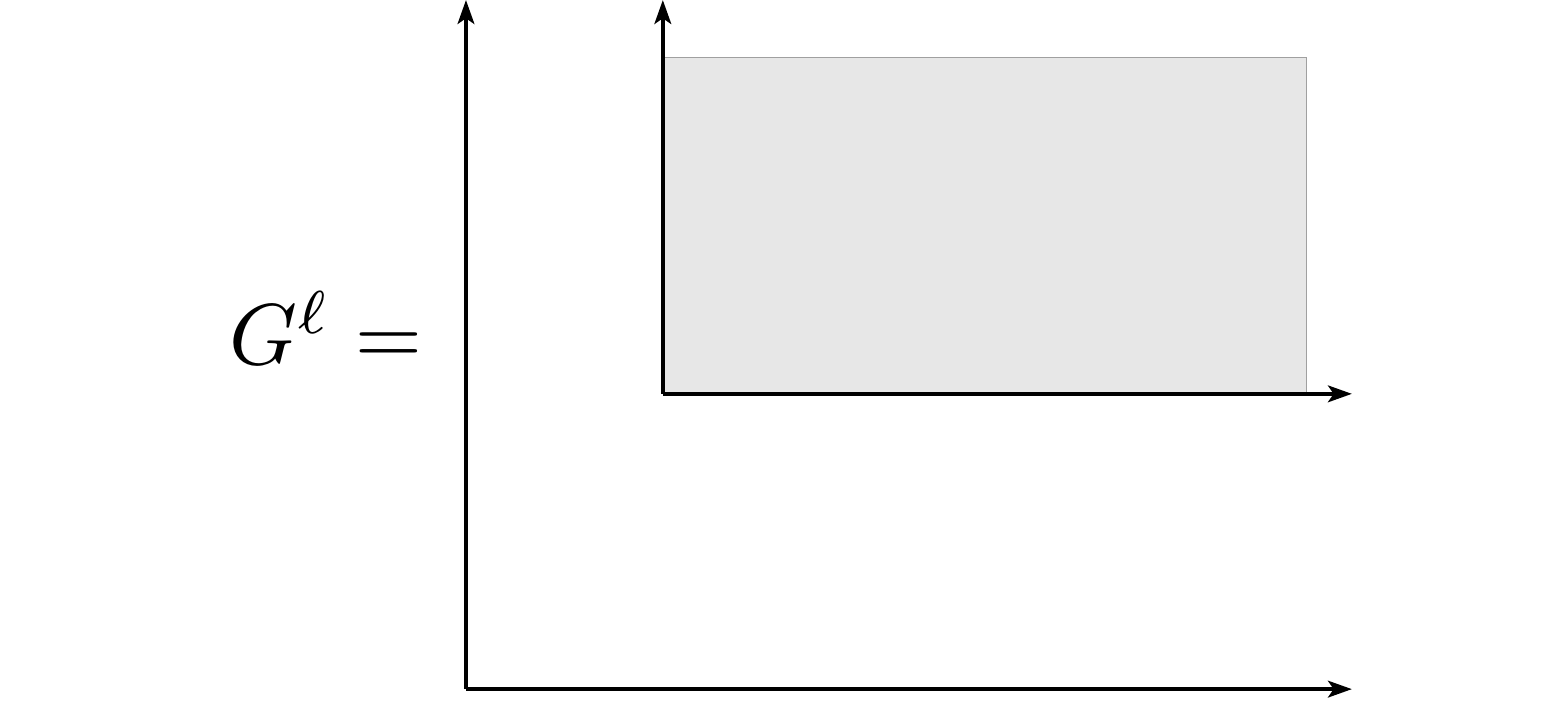}}
\end{center}

\begin{center}
\scalebox{0.4}{\includegraphics[width=\linewidth]{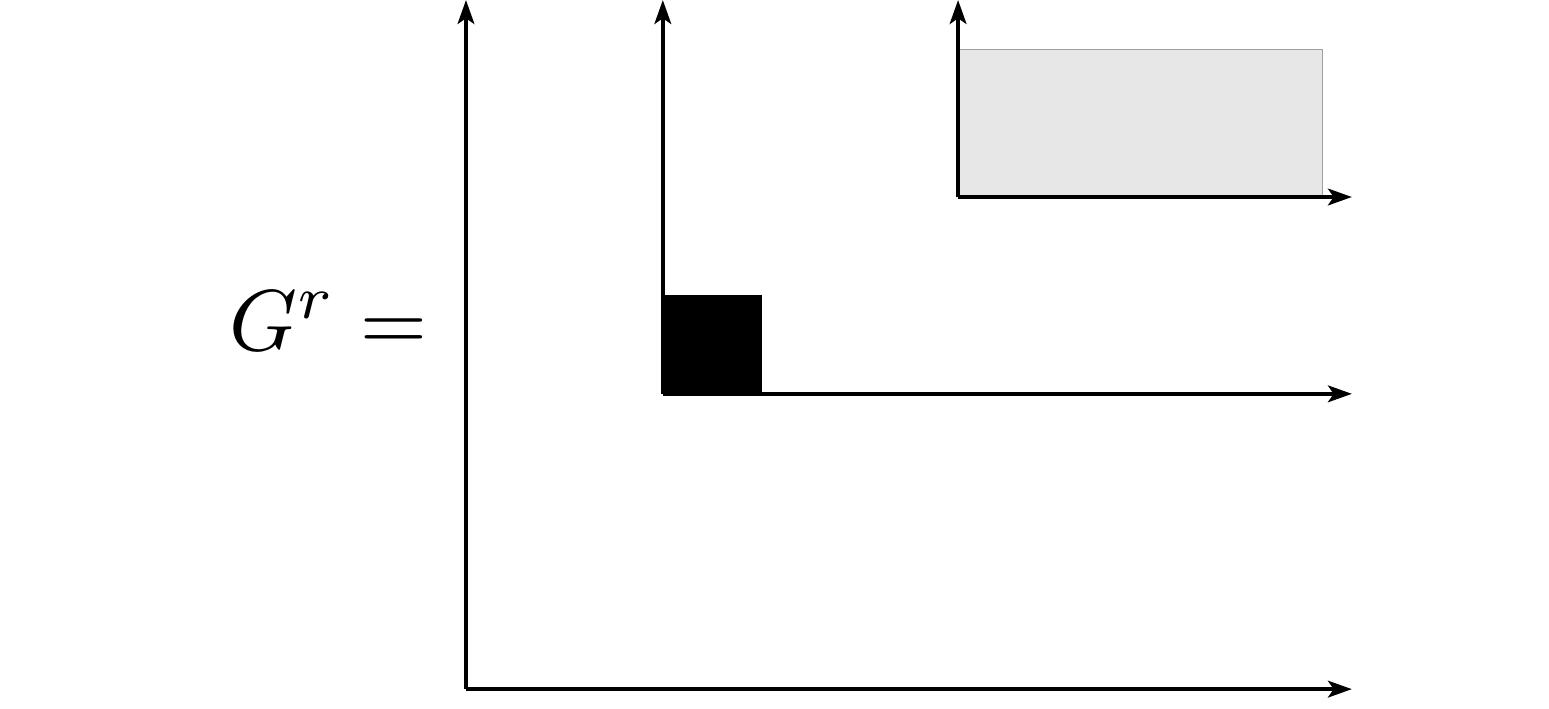}}
\caption{F1  argument in {\sc polychromatic chomp}.}\label{fig3}
\end{center}
\end{figure}

Hence, all $(G^L,G^R)$ satisfy one or both of the properties. Therefore, by Lemma \ref{th:sc}, all {\sc polychromatic chomp} positions are numbers.
\end{example}

\begin{example}\label{ex:brhssp}
Let $G$ be a {\sc blue-red-hackenbush} string and let $\ell$ and $r$ be the edges played by Left and Right, respectively. If $r$ is higher up the string than $\ell$ then playing $\ell$ eliminates $r$. Thus, $G^{r \ell}$ and $G^\ell$ are identical. Otherwise, playing $r$ eliminates $\ell$, and $G^{\ell r}\cong G^r$.

Hence, by Theorem \ref{th:snc}, all {\sc blue-red-hackenbush} strings are numbers.
\end{example}

\begin{example}
Consider a $m\times n$ {\sc cutcake} position. The moves are not independent but almost so.
For given $\ell$ and $r$, consider the pair of options,
$G^L=m\times(n-\ell) + m\times \ell$
and $G^R=(m-r)\times n + r\times n$, and their options
\begin{eqnarray*}
G^{LR}&=&m\times(n- \ell)+(m-r)\times \ell +  r\times\ell,\\
G^{RL}&=&(m-r)\times n  + r\times (n-\ell)+ r\times \ell.
\end{eqnarray*}

The moves cannot be interchanged and get the same board position. However, we know that
if $i>j$ then $k\times i \geqslant k\times j$ (intuitively, there are more moves for Left in $k\times i$ than in $k\times j$ )
and similarly $i\times k \leqslant j\times k$. The terms of $G^{LR}$ and $G^{RL}$ pair off:
$r\times\ell$ is in both;
$(m-r)\times \ell \leqslant (m-r)\times n$; $m\times(n-\ell)  \leqslant r\times(n-\ell)$.
Therefore $G^{LR} \leqslant G^{RL}$ and, so, $(G^L,G^R)$ satisfies the F2  property. Therefore, by Theorem \ref{th:i},
$G$ is an integer.
\end{example}

\begin{example}
\begin{enumerate}
    \item Given a {\sc shove} position $G$, if any token is pushed off the end of the strip then $(G^\ell,G^r)$ satisfies the F1  property; if not, $(G^\ell,G^r)$ satisfies the F2  property.
    \item Given a {\sc push} position $G$, if any token pushes the other then $(G^\ell,G^r)$ satisfies the F1  property; if not, $(G^\ell,G^r)$ satisfies the F2  property.
    \item Given a {\sc lenres} position $G$, if any digit in the move replaces the other then $(G^\ell,G^r)$ satisfies the F1  property; if not, $(G^\ell,G^r)$ satisfies the F2  property.
      \item Given a {\sc domino shave} position $G$, $(G^L,G^R)\in G^\mathcal{L}\times G^\mathcal{R}$
      satisfies the F1  property.
     \item Given a {\sc divisors} position $G=(l,r)$ and a pair of options, $(G^L,G^R)=((\ell',r),(\ell,r'))$, if $\ell'=r$ or $\ell=r'$ then $(G^L,G^R)$ satisfies the F1  property; if not, $(G^L,G^R)$ satisfies the F2  property.
          \item Given a {\sc partizan turning turtles} position $G$, if $G^L$ and $G^R$ conflict then $(G^L,G^R)$ satisfies the F1  property; if $G^L$ and $G^R$ don't conflict then $(G^L,G^R)$ satisfies the F2  property.
\end{enumerate}
\end{example}

Even more can be said about {\sc blue-red-cherries} \cite{ANW} and {\sc erosion} \cite{ANW}. These game all satisfy the F2  property and thus, by Theorem  \ref{th:i}, they are integers.
\begin{example}
\begin{enumerate}
  \item Given a {\sc blue-red-cherries} position $G$, all $(G^L,G^R)\in G^\mathcal{L}\times G^\mathcal{R}$.
  If Left removes a cherry from one end ($\ell$) and Right removes a cherry  from the other end ($r$)
  then $G^{\ell r} \cong G^{r\ell}$.
      \item Given an {\sc erosion} position $G$, all $(G^L,G^R)\in G^\mathcal{L}\times G^\mathcal{R}$ satisfy the F2  property. This is vacuously true since, by the rules, it is impossible for both players to have options at the same time.
\end{enumerate}
\end{example}

For the values all to be numbers, the properties must always be true. It is not sufficient for most of the positions to satisfy them. Two games, that have $\N$-positions but where many of the positions naturally satisfy one or the other property, are:
\begin{enumerate}
\item F1 : In \textsc{partizan euclid}~\cite{MN} with $G=(p,q)$ and $p>2q$ then  $G^{LR}=G^R$ or $G^{RL}=G^L$.
\item F2 : In the \textsc{partizan subtraction} subset of \textsc{splittles}~\cite{Mesdal}, let $a$ be the largest that can be taken. Suppose the heap size is $n$, and $n\geqslant 2a$ then Left taking $\ell$ and Right taking $r$ results in a heap of size $n-\ell-r$ regardless of the order. Thus $G^{\ell r}\cong G^{r\ell}$.
\end{enumerate}
\section{Proofs} \label{sec:c}

\begin{theorem} (Outcomes and numbers)\label{th:con}
Let $\mathbf{S}$ be a {\rm HCR}. All positions $G\in \mathbf{S}$ are numbers if and only if there is no $G\in \mathbf{S}$ 
 and a number $x$ such that $G+x\in\mathcal{N}$.
\end{theorem}

\begin{proof} $\,$\\
($\Rightarrow$) If all positions $G \in \mathbf{S}$ are numbers then, regardless of what the numbers $x$ are, all $G+x$ are numbers. Hence,  there is no $G\in\mathbf{S}$ and a number $x$ such that $G+x\in\mathcal{N}$.\\

\noindent
($\Leftarrow$) Let $G\in\mathbf{S}$. 
 If $G^\mathcal{L}=\emptyset$ or $G^\mathcal{R}=\emptyset$ then $G$ is an integer. Suppose that  $G^\mathcal{L}\neq\emptyset$ and $G^\mathcal{R}\neq\emptyset$. By induction, since $\mathbf{S}$ is hereditary closed, all $G^L\in G^\mathcal{L}$ and $G^R\in G^\mathcal{R}$ are numbers. Hence, after removing dominated options, there are 3 possible cases:

\begin{enumerate}
  \item[1)] $G=\{a\,|\,a\}=a+*$, where $a$ is a number;
  \item[2)] $G=\{a\,|\,b\}=\frac{a+b}{2}\pm\frac{a-b}{2}$, where $a$ and $b$ are numbers and $a>b$;
  \item[3)] $G=\{a\,|\,b\}$, where $a$ and $b$ are numbers and $a<b$.
\end{enumerate}

If 1) or 2) then, $G-a\in\mathcal{N}$ or $G-\frac{a+b}{2}\in\mathcal{N}$, contradicting the assumptions. Therefore, we must have 3). Now $G$ is the simplest number strictly between $a$ and $b$. \\

In all cases $G$ is a number and the theorem follows.
\end{proof}

Now, a natural question arises: Is it easy, in practice, to know if a HCR does not have positions such that $G+x\in\mathcal{N}$? In other words, is Theorem \ref{th:con} useful? Lemma \ref{th:sc} answers that question.\\

\begin{lemma} \label{th:sc}
Let $\mathbf{S}$ be a {\rm HCR}. If, for any position $G\in \mathbf{S}$, all pairs $(G^L,G^R)\in G^\mathcal{L}\times G^\mathcal{R}$ satisfy the F1  property or the F2  property, then  all positions $G\in\mathbf{S}$ are numbers.
\end{lemma}

\begin{proof}$\,$\\

\vspace{-0.4cm}
Item 1: By Theorem \ref{th:con}, it is enough to prove that if all pairs $(G^L,G^R)\in G^\mathcal{L}\times G^\mathcal{R}$ satisfy the F1  property or the F2  property then there is no $G\in\mathbf{S}$ and a number $x$ such that $G+x\in\mathcal{N}$.\\

For the contrapositive, suppose that there is a position $G\in\mathbf{S}$ and a number $x$  such that $G+x\in\mathcal{N}$. Assume that the birthday of $G$ in such conditions is the smallest possible.\\

Since  $G+x\in\mathcal{N}$, there are $G^L+x\geqslant 0$ and $G^R+x \leqslant 0$ (Theorem \ref{th:nat}). Due to the hypothesis, the pair $(G^L,G^R)$ satisfies the F1  property or the F2  property. If the pair satisfies the F1  property, there is $G^{RL}$ such that $G^{RL}\geqslant G^L$ or there is $G^{LR}$ such that $G^{LR}\leqslant G^R$. If the first happens, then $G^{RL} \geqslant G^L$ implies $G^{RL}+x
\geqslant G^L+x \geqslant 0$. That is  incompatible with $G^R+x\leqslant 0$. If the second happens, then $G^{LR}\leqslant
 G^R$ implies $G^{LR}+x \leqslant G^R+x \leqslant 0$. That is incompatible with $G^L+x
\geqslant 0$. In either case we have a contradiction; the pair $(G^L,G^R)$ cannot satisfy the F1  property.\\

Hence,  the pair $(G^L,G^R)$ satisfies the F2  property, and there are $G^{LR}\in\mathbf{S}$ and $G^{RL}\in\mathbf{S}$ such that $G^{LR}\leqslant G^{RL}$. Since $G^L+x \geqslant 
0$, we have $G^{LR}+x \not \leqslant 0$. Also, since $G^R+x\leqslant 0$, we have $G^{RL}+x \not \geqslant
0$. The second inequality allows to conclude that $G^{LR}+x \not \geqslant
0$ because $G^{LR}\leqslant G^{RL}$. However, $G^{LR}+x \not \leqslant 0$ and $G^{LR}+x \not \geqslant 0$, implies that  $G^{LR}+x\in\mathcal{N}$, contradicting the smallest rank assumption. Therefore, the pair $(G^L,G^R)$ cannot satisfy the F2  property.\\

The pair $(G^L,G^R)$ doesn't satisfy the F1 property or the F2
property, and that contradicts the hypothesis. There is no $G\in\mathbf{S}$ and  number $x$ such that $G+x\in\mathcal{N}$. Therefore, all positions $G\in\mathbf{S}$ are numbers.\\
\end{proof}

It is possible to have a pair of options that satisfies the F2  property without satisfying the F1  property; an example of that is a pair like $(G^L,G^R)=(\{0,*\,|\,*\},\{*\,|\,0,*\})$. However, the options of $*$ do not satisfy the F2  property. On the other hand, if all followers also satisfy the F2  property, then the next lemma shows that all pairs satisfy the F1  property.

\begin{lemma}\label{lem:descendfloss} Let $\mathbf{S}$ be a {\rm HCR} such that, given any position $G\in \mathbf{S}$, all pairs $(G^L,G^R)\in G^\mathcal{L}\times G^\mathcal{R}$ satisfy the F2  property then all pairs satisfy the F1  property.
\end{lemma}
\begin{proof}

It is enough to prove that if a pair $(G^L,G^R)\in G^\mathcal{L}\times G^\mathcal{R}$ satisfies the F2  property, it also satisfies the F1  property.\\

Suppose that a pair  $(G^L,G^R)$ satisfies the F2  property. If so, by definition, there are $G^{LR}$ and $G^{RL}$ such that $G^{LR}\leqslant
G^{RL}$. Since $G^{LR}$ is a right option of $G^L$, we have $G^{LR} \not \leqslant
G^L$. On the other hand, by Lemma \ref{th:sc}, all positions of $\mathbf{S}$ are numbers, so $G^{LR}$ cannot be incomparable with $G^L$. Therefore, we must have $G^{LR}> G^L$ and consequently $G^{RL}\geqslant G^{LR}>G^L$. Thus  $(G^L,G^R)$ satisfies the F1  property.
\end{proof}

\begin{theorem} \label{th:snc}
Let $\mathbf{S}$ be a {\rm HCR}. All positions $G\in\mathbf{S}$ are numbers if and only if, for any position $G\in \mathbf{S}$, all pairs $(G^L,G^R)\in G^\mathcal{L}\times G^\mathcal{R}$ satisfy the F1  property.
\end{theorem}

\begin{proof}
($\Leftarrow$) Consequence of Lemma \ref{th:sc}.\\

($\Rightarrow$) Let $G\in \mathbf{S}$ and $(G^L,G^R)\in G^\mathcal{L}\times G^\mathcal{R}$. Since $G$ is a number, $G^L< G^R$ and thus $G^L-G^R< 0$. Since Right, playing first, wins $G^L-G^R$, either there is a $G^{LR}$ with $G^{LR}-G^R\leqslant 0$ or some $G^{RL}$ with $G^{L}-G^{RL}\leqslant 0$. Hence, there is $G^{LR}\leqslant G^R$ or $G^{L}\leqslant G^{RL}$, and, by definition, $(G^L,G^R)$ satisfies the F1  property.
\end{proof}

By Lemma \ref{lem:descendfloss}, if all pairs $(G^L,G^R)$ satisfy the F1  property or the F2  property, then all pairs  satisfy the F1  property. That means that a pair satisfying the F2  property also satisfies the F1  property. But, observe that the opposite is not true: it is possible to have a pair satisfying the F1  property without satisfying the F2  property. For example, if $G=\frac{1}{2}=\{0\,|\,1\}$ (canonical form), then the pair $(0,1)$ satisfies the F1  property and does not satisfy the F2  property because $G^{L\mathcal{R}}=\emptyset$. The F2  property is a stronger condition and has a surprising consequence.

\begin{theorem} \label{th:i}
Let $\mathbf{S}$ be a {\rm HCR}. If, for any position $G\in \mathbf{S}$, all pairs $(G^L,G^R)\in G^\mathcal{L}\times G^\mathcal{R}$ satisfy the F2  property, then all positions $G\in\mathbf{S}$ are integers.
\end{theorem}

\begin{proof}

Let $G\in \mathbf{S}$. If $G\mathcal{^L}=\emptyset$ or $G\mathcal{^R}=\emptyset$, then $G$ is an integer and the theorem holds. Suppose that $G\mathcal{^L}\not=\emptyset$ and $G\mathcal{^R}\not=\emptyset$. By Lemma \ref{th:sc}, all positions $G\in\mathbf{S}$ are numbers and the canonical form is $G=\{G^L\mid G^R\}$.
 By induction, $G^L$ and $G^R$ are integers. Since $(G^L,G^R)$ satisfies the F2  property, there is $G^{LR}\leqslant
G^{RL}$ and, by induction, both are integers. Let $G^{RL}=k$.
  Therefore, we have that $G^L<
 k$ and $G^R>
 k$, and, thus $G$ is an integer. \end{proof}

\begin{observation}
Theorem \ref{th:i} exhibits a sufficient but \emph{not necessary} condition. Consider $\mathbf{S}$, a HCR whose game forms are $\{-2\,|\,0\}$, $-2$, $-1$, and $0$ (the last three, canonical forms). Of course, all game values of $\mathbf{S}$ are integers. However, regarding $G=\{-2\,|\,0\}$, the pair $(G^L,G^R)=(-2,0)$ does not satisfy the F2 property.
\end{observation}

\section*{Acknowledgments}
\noindent
Alda Carvalho was partially supported by the Project CEMAPRE/REM-UIDB/05069/2020, financed by FCT/MCTES through national funds.\\

\noindent
Melissa A. Huggan was supported by the Natural Sciences and Engineering Research Council of Canada (funding reference number PDF-532564-2019).\\

\noindent
Richard J. Nowakowski was supported by the Natural Sciences and Engineering Research Council of Canada (funding reference number 4139-2014).\\

\noindent
Carlos Santos is a CEAFEL member and has the support of
UID/MAT/04721/2019 strategic project.\\

\vspace{1cm}
\textbf{Appendix 1: Rulesets}\label{appendix1}\\

\noindent
{\sc divisors}\\

\noindent
Position: An ordered pair of positive integers $(l,r)$.\\

\noindent
Moves: Left is allowed to replace $(l,r)$ by $(l',r)$  where $l'<l$ is a divisor of $r$. Right is allowed to replace $(l,r)$ by $(l,r')$  where $r'<r$ is a divisor of $l$.

\begin{center}
$(5,4)$ \raisebox{0.0cm}{$\overset{L}{\rightarrow}$} $(2,4)$  \raisebox{0.0cm}{$\overset{R}{\rightarrow}$} $(2,1)$ \raisebox{0.0cm}{$\overset{L}{\rightarrow}$} $(1,1)$
\end{center}

\vspace{0.5cm}
\noindent
{\sc partizan turning turtles}\\

\noindent
Position: A line of turtles. A turtle may be on its feet or on its back.\\

\noindent
Moves: Left is allowed to choose two upside-down turtles and turn them onto its feet. Right is also allowed to choose a pair of turtles, provided that the leftmost is on its feet and the other is on its back; his move is turning over both turtles.

\begin{center}
\td\tu\td\td \raisebox{0.1cm}{$\overset{L}{\rightarrow}$} \tu\tu\td\tu  \raisebox{0.1cm}{$\overset{R}{\rightarrow}$} \tu\td\tu\tu
\end{center}

\vspace{0.5cm}
\noindent
{\sc polychromatic chomp}\\

\noindent
Position: A grid with one poison square in the lower left corner. Besides the poison square, each square is either black or gray.\\\\

\noindent
Moves: On her turn, Left chooses a black square and removes it and all other squares above or to the right of it. On his turn, Right moves analogously, but he has to choose a gray square.

\begin{center}
\includegraphics[width=\linewidth]{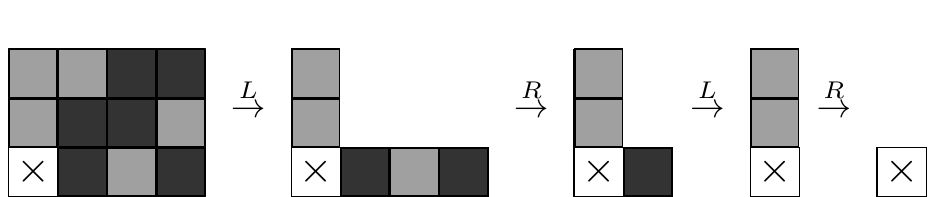}
\end{center}

\end{document}